\documentclass[11pt,reqno]{amsart}

\usepackage[T1]{fontenc}
\usepackage{lmodern}
\usepackage{amsmath,amssymb,amsthm,mathtools,comment}
\usepackage{microtype}
\usepackage{enumitem}
\usepackage[colorlinks=true,linkcolor=blue,citecolor=blue,urlcolor=blue]{hyperref}

\newtheorem{theorem}{Theorem}[section]
\newtheorem{proposition}[theorem]{Proposition}
\newtheorem{lemma}[theorem]{Lemma}
\newtheorem{corollary}[theorem]{Corollary}
\theoremstyle{remark}
\newtheorem{remark}[theorem]{Remark}
\newtheorem{notation}[theorem]{Notation}

\newcommand{\PP}{\mathbb P}
\newcommand{\ZZ}{\mathbb Z}

\newcommand{\OO}{\mathcal{O}}

\newcommand{\Pic}{\operatorname{Pic}}
\newcommand{\Bl}{\operatorname{Bl}}
\newcommand{\Nef}{\operatorname{Nef}}

\newcommand{\Proj}{\operatorname{Proj}}

\newcommand{\Sbar}{\overline{S}}
\newcommand{\Et}{\widetilde{E}}

\title{A Sarkisov Involution from a Rational Cubic Curve on a Quartic Double Solid}

\author{Joseph Cutrone}
\address{Department of Mathematics, Johns Hopkins University, 3400 N. Charles St, Baltimore, MD 21218}
\email{jcutron2@jhu.edu}

\author{Nicholas Marshburn}
\address{Department of Mathematics, Johns Hopkins University, 3400 N. Charles St, Baltimore, MD 21218}
\email{nmarshb1@jhu.edu}

\date{}

\begin{document}

\begin{abstract}
Building on an earlier classification of smooth weak Fano threefolds of Picard number two with small anticanonical morphism, we prove that Case~28 of the table of E1-E1 links from that classification is geometrically realizable. We construct the link by blowing up a smooth rational cubic curve on a degree-two del Pezzo threefold $V_2$, realized as a quartic double solid. A lattice calculation on an anticanonical K3 surface shows that the anticanonical morphism of the blow-up is small, and the covering involution of the anticanonical double cover realizes the flop. The resulting Sarkisov link is a birational involution with the numerical invariants of E1-E1 Case~28.
\end{abstract}

\maketitle

\section{Introduction}

In \cite{CM25}, the authors provided an update to the classification of
smooth weak Fano threefolds of Picard number two with small anticanonical
morphism. That paper reduced the remaining open cases on the E1-E1
table to Cases 28, 59, 61, and 80. The purpose of the present paper is
to prove the geometric existence of Case 28.

The numerical classification of these Sarkisov links was initiated in
\cite{CM13}, building on the work of \cite{JPR07}, and the geometric
existence of many of the remaining cases was subsequently settled in
\cite{ACM17, CLM19, CM25}; related constructions of Sarkisov links by blowing up curves on Fano threefolds appear in \cite{BL12,BL15}. Case 28 remained open. Its numerical invariants are:
\begin{equation}\label{eq:case28data}
 (-K_X)^3=2,\qquad (-K_Y)^3=(-K_{Y^+})^3=16,
\end{equation}
\begin{equation}\label{eq:curve-data}
 (r,d,g)=(r^+,d^+,g^+)=(2,3,0),
\end{equation}
with
\begin{equation}\label{eq:table-defect}
 \alpha=8,\qquad \beta=-1,\qquad \frac{e}{r^3}=69.
\end{equation}

\begin{notation}
Recall the conventions of \cite{CM13,CM25} for a Sarkisov link of type
$E_1$-$E_1$: $Y$ is a Fano threefold of Picard number one and Fano index
$r$ (so $-K_Y\sim rH$ for the primitive ample generator $H$), and $C\subset Y$
is a smooth curve of degree $d=H.C$ and genus $g$ blown up to produce
$X$; the data $(r^+,d^+,g^+)$ on the other side of the flop
$\chi\colon X\dashrightarrow X^+$ are defined symmetrically. Let $E^+\subset
X^+$ be the exceptional divisor of the second $E_1$-contraction, and let
$D\subset X$ be its strict transform under $\chi^{-1}$. Write
\[
D=\alpha(-K_X)+\beta E.
\]
If $\widetilde E\subset X^+$ denotes the strict transform of $E$ under
the flop, the defect of the flop is
\[
e=E^3-\widetilde E^3,
\]
and the normalized defect appearing in the table is $e/r^3$. Throughout,
all varieties are defined over the field $\mathbb{C}$.
\end{notation}

Thus Case 28 should arise from the blow-up of a smooth rational cubic curve on a degree-two del Pezzo threefold $V_2$, with the same numerical
geometry on the other side of the flop.

As noted in \cite{CM25}, Case 28 is different from the other remaining open cases because the numerical link begins with the degree-two del Pezzo threefold $V_2$. The methods of \cite{ACM17} do not apply directly to this case, since $V_2$ does not admit the projective complete intersection model used there. We instead realize $V_2$ as a double
cover
\[
 V_2\longrightarrow \PP^3
\]
branched over a smooth quartic surface.\footnote{The authors would like to thank V.~Cheltsov for suggesting the use of the quartic double solid model during discussions surrounding the SAGA$^2$ Conference at the University of French Polynesia in March 2026.}

To summarize, the construction begins with a twisted cubic $\Gamma\subset\PP^3$ contained in a smooth quartic K3 surface of Picard number two. We choose the branch quartic of the double cover so that the inverse image of $\Gamma$ splits into two rational cubic curves. Blowing up one of these curves gives the desired weak Fano threefold. A lattice computation on the strict transform of the K3 surface shows that the anticanonical morphism is small.

Our main result is the following:
\begin{theorem}\label{thm:main}
The numerical invariants listed in the E1-E1 table for Case 28 are geometrically realizable as a Sarkisov link of Fano threefolds.
\end{theorem}

More precisely, we construct a smooth rational cubic curve $C$ on a degree-two del Pezzo threefold $Y = V_2$ such that $X = \Bl_C(Y)$ is a smooth weak Fano threefold of Picard number two with small anticanonical morphism. We then show that the flop of $X$ is realized by a birational involution of $X$, induced by the covering involution of the anticanonical double cover, so that the contraction on the other side of the flop is again the blow-up of a smooth rational cubic curve on $V_2$ and the resulting Sarkisov link is a birational involution of $Y$ with the numerical invariants listed in Case~28.\\

The proof is given in Sections ~2-4.

%%%%%%%%%%%%%%%%%%%%%%%%%%%%%%%%%%%%%%%
\section{Construction of Case 28}
\subsection{A quartic K3 surface containing a twisted cubic}\label{sec:k3}

Fix a twisted cubic $\Gamma\subset\PP^3.$
Its homogeneous ideal is generated by three quadrics. Choose generators
\begin{equation}\label{eq:twisted-ideal}
 I_\Gamma=(Q_1,Q_2,Q_3).
\end{equation}
We will write $h$ for the hyperplane class on any surface in $\PP^3$ under consideration. We use the following form of the Noether-Lefschetz theorem for surfaces containing a prescribed curve \cite[Theorem~II.3.1]{Lopez91}: if
$\Gamma\subset\mathbb{P}^3$ is a smooth connected curve, $d\ge4$, and $\mathcal{I}_\Gamma(d-1)$ is globally generated, then a very general member $S\in|\mathcal{I}_\Gamma(d)|$ is smooth and $\operatorname{Pic} (S)=\mathbb{Z}h\oplus\mathbb{Z}\Gamma$.

\begin{lemma}\label{lem:k3existence}
There exists a smooth quartic surface $S_0\subset\PP^3$ containing $\Gamma$ such that
\begin{equation}\label{eq:PicS}
 \Pic(S_0)=\ZZ h\oplus\ZZ\Gamma.
\end{equation}
Moreover,
\begin{equation}\label{eq:lattice}
 h^2=4,\qquad h\cdot\Gamma=3,\qquad \Gamma^2=-2.
\end{equation}
\end{lemma}

\begin{proof}
Since the homogeneous ideal of a twisted cubic is generated by quadrics, $\mathcal{I}_\Gamma(2)$, and a fortiori $\mathcal{I}_\Gamma(3)$, is globally generated. Lopez's theorem with $d=4$ gives \eqref{eq:PicS}. Since $S_0$ is a quartic K3 surface, $h^2=4$. The equality $h\cdot\Gamma=3$ is the degree of the twisted cubic. Finally adjunction on the K3 surface gives
\[
 2p_a(\Gamma)-2=\Gamma^2+K_{S_0}\cdot\Gamma=\Gamma^2.
\]
As $\Gamma\simeq\PP^1$, we obtain $\Gamma^2=-2$.
\end{proof}

\begin{comment}
\begin{lemma}\label{lem:quadric}
There exists a quadric $Q\subset\PP^3$ such that $Q|_\Gamma$ has six distinct zeros.
\end{lemma}

\begin{proof}
Since $\Gamma\subset\PP^3$ is a twisted cubic, we have
$\Gamma\simeq\PP^1$ and $\OO_\Gamma(1)\simeq\OO_{\PP^1}$
and therefore $\OO_\Gamma(2)\simeq\OO_{\PP^1}(6).$
Consider the restriction map
\[
H^0(\PP^3,\OO_{\PP^3}(2)) \longrightarrow H^0(\Gamma,\OO_\Gamma(2)) \simeq H^0(\PP^1,\OO_{\PP^1}(6)).
\]
Since the homogeneous ideal of $\Gamma$ is generated by three independent quadrics, 
\[
h^0(\PP^3,\mathcal I_\Gamma(2))=3.
\]
Moreover,
\[
h^0(\PP^3,\OO_{\PP^3}(2))=10,
\qquad
h^0(\PP^1,\OO_{\PP^1}(6))=7.
\]
Hence the restriction map is surjective. Thus, as $Q$ varies among quadrics in
$\PP^3$, the section $Q|_\Gamma$ varies over all sections of
$\OO_{\PP^1}(6)$.

The locus of sections of $\OO_{\PP^1}(6)$ having a multiple zero is the
discriminant hypersurface in $H^0(\PP^1,\OO_{\PP^1}(6))$. Its complement is
therefore a nonempty Zariski-open subset. For example, the binary sextic
$s^6-t^6$ has six distinct zeros over $\mathbb C$. Consequently, for a general
quadric $Q$, the section $Q|_\Gamma$ has six distinct zeros.
\end{proof}
\end{comment}

We next choose a general quadric $Q$ that meets $\Gamma$ in six distinct points. For $G\in H^0(\PP^3,\mathcal I_\Gamma(4))$, set $S_0 :=\{G=0\}$ and 
\begin{equation}\label{eq:branch}
 B_Q:=\{Q^2+G=0\}\subset\PP^3.
\end{equation}

\begin{lemma}\label{lem:branchsmooth}
There exists a quartic $G\in H^0(\PP^3,\mathcal I_\Gamma(4))$ such that $S_0$ is smooth, $\Pic(S_0)=\ZZ h\oplus\ZZ\Gamma,$ and the quartic surface $B_Q$ is smooth.
\end{lemma}

\begin{proof}
Let $V=H^0(\mathbb{P}^3,\mathcal{I}_\Gamma(4))$ and consider the linear system
\[
\Lambda=\mathbb{P}\bigl(\langle Q^2\rangle+V\bigr)\subset|\mathcal{O}_{\mathbb{P}^3}(4)|.
\]
Since $\mathcal{I}_\Gamma(4)$ is globally generated, the common zero locus of
all sections of $V$ is $\Gamma$. Hence
\[
\operatorname{Bs}(\Lambda)=\Gamma\cap\{Q=0\},
\]
which consists of six distinct points.

By Bertini's theorem, a general member of $\Lambda$ is smooth away from its
base locus. On the open subset of $\Lambda$ where the coefficient of $Q^2$ is
nonzero, we may rescale the equation and write a general member as
\[
Q^2+G=0,
\qquad G\in V.
\]
Thus, for $G$ in a nonempty Zariski-open subset of $V$, the corresponding
surface $B_Q$ is smooth away from the six points of
$\Gamma\cap\{Q=0\}$.

By the Noether-Lefschetz statement used in Lemma~\ref{lem:k3existence}, for a
very general $G\in V$ the surface $S_0=\{G=0\}$ is smooth and satisfies $\Pic(S_0)=\ZZ h\oplus\ZZ\Gamma.$ We may therefore choose $G$ satisfying this condition and the preceding
Bertini condition simultaneously.

It remains to check smoothness of $B_Q$ at the six base points. Let
$p\in\Gamma$ satisfy $Q(p)=0$, and put $F=Q^2+G.$

Then $dF_p=2Q(p)dQ_p+dG_p=dG_p.$ Since $S_0=\{G=0\}$ is smooth, $dG_p\neq0$. Hence $dF_p\neq0$, so $B_Q$ is
smooth at $p$. Thus $B_Q$ is smooth everywhere.
\end{proof}

Fix $Q$ and $G$ as in Lemma ~\ref{lem:branchsmooth}.
Since \eqref{eq:twisted-ideal} generates $I_\Gamma$, its degree-$4$ piece is
generated by the products of the $Q_i$ with quadrics. Thus there exist quadrics
$A_1,A_2,A_3$ such that $G=Q_1A_1+Q_2A_2+Q_3A_3.$

\subsection{The quartic double solid and the rational cubic curve}\label{sec:double-solid}

Let
\begin{equation}\label{eq:Yweighted}
 Y=\{w^2=Q^2+G\}\subset\PP(1,1,1,1,2),
\end{equation}
where $w$ has weight $2$. Projection to the weight-$1$ coordinates gives a double cover $q\colon Y\longrightarrow\PP^3$ branched over $B_Q$. Since $B_Q$ is smooth, $Y$ is smooth. This is the standard degree-$2$ del Pezzo threefold (quartic double solid). Set $ H=q^*\OO_{\PP^3}(1).$
Then
\begin{equation}\label{PicY}
 \Pic(Y)=\ZZ H,\qquad -K_Y=2H,\qquad H^3=2,
\end{equation}
and hence $(-K_Y)^3=16$; see, for example, \cite[Chapter~12]{IP99}.

Over $\Gamma$, the quartic $G$ vanishes identically, and therefore the equation of $Y$ becomes $w^2=Q^2.$
Thus the inverse image of $\Gamma$ splits into two curves
\begin{equation}\label{eq:splitcurves}
 q^{-1}(\Gamma)=C\cup C^-, \,\, C=\{x\in\Gamma,\ w=Q(x)\},\,\, C^-=\{x\in\Gamma,\ w=-Q(x)\}.
\end{equation}
Both project isomorphically to $\Gamma$, and hence
\begin{equation}\label{eq:Cdata}
 C\simeq C^-\simeq\PP^1,
 \qquad H\cdot C=H\cdot C^-=3.
\end{equation}
In particular, $C$ is a smooth rational cubic curve in the sense of the E1 tables.

\subsection{The blow-up and its anticanonical system}\label{sec:blowup}

Let
$$
 \pi\colon X=\Bl_C(Y)\longrightarrow Y
$$
be the blow-up, and let $E$ be its exceptional divisor. Since $Y$ and $C$ are smooth, $X$ is smooth, and $
 \Pic(X)=\ZZ H\oplus\ZZ E, \rho(X)=2$, where we continue to write $H$ for $\pi^*H$. The canonical divisor formula gives $-K_X=2H-E.$

\subsection{Intersection numbers}
The standard intersection formulas for the blow-up of a smooth threefold
along a smooth curve, together with
$$\deg N_{C/Y}=(-K_Y).C+2g(C)-2=2H.C-2=6-2=4,$$
give the basic intersection numbers
\begin{equation}\label{eq:HEtable}
 H^3=2,\,\,\, H^2E=0,\,\,\, HE^2=-H.C=-3,\,\,\, E^3=-\deg N_{C/Y}=-4.
\end{equation}
Expanding $-K_X=2H-E$ against \eqref{eq:HEtable} gives
\begin{align}
 (-K_X)^3  &=8H^3-12H^2E+6HE^2-E^3=2,  \label{eq:-K_X3}\\
 (-K_X)^2.H &=4H^3-4H^2E+HE^2=5,    \label{eq:-K_X2H}\\
 (-K_X)^2.E &=4H^2E-4HE^2+E^3=8,    \label{eq:-K_X2E}\\
 (-K_X).E^2 &=2HE^2-E^3=-2.       \label{eq:-K_XE2}
\end{align}

\section{The Anticanonical Morphism}
\subsection{Global generation}

Put $t=w-Q.$ The defining equation of $Y$ and $G$ give
\begin{equation}\label{eq:keyrelation}
 t(t+2Q)=w^2-Q^2=G=Q_1A_1+Q_2A_2+Q_3A_3.
\end{equation}
The curve $C$ is scheme-theoretically cut out on $Y$ by
\begin{equation}\label{eq:idealC}
 \mathcal I_{C/Y}=(t,Q_1,Q_2,Q_3).
\end{equation}
Indeed, imposing $Q_1=Q_2=Q_3=0$ places the image point on $\Gamma$, and imposing $t=0$ chooses the component $w=Q$ in \eqref{eq:splitcurves}. Each generator in \eqref{eq:idealC} is a section of $\OO_Y(2H)$. Therefore 
\begin{equation}\label{eq:globalgenI}
 \mathcal I_{C/Y}(2H)
\end{equation}
is globally generated. By equations \eqref{eq:-K_X3} and \eqref{eq:globalgenI}, we have proved the following:

\begin{proposition}\label{prop:weakfano}
The anticanonical divisor $-K_X$ is globally generated, nef and big. In particular $X$ is a smooth weak Fano threefold, and $(-K_X)^3=2$.
\end{proposition}

\subsection{The anticanonical morphism}\label{sec:conjugate}
We now show that the anticanonical morphism is nontrivial and small. Let $R\subset X$ be the strict transform of $C^-$. We compute the intersection of $R$ with $E$. On $C^-$ we have
$Q_i=0$ and $w=-Q$, and hence
\[
t|_{C^-}=(w-Q)|_{C^-}=-2Q|_\Gamma.
\]
Thus the scheme-theoretic intersection $C\cap C^-$ is defined on
$\Gamma$ by $Q|_\Gamma=0$. By our choice of $Q$, this is a reduced
scheme of length six.

Since $C\cap C^-$ is an effective Cartier divisor on the smooth curve
$C^-$, the blow-up of $C^-$ along $C\cap C^-$ is isomorphic to
$C^-$. Hence $\pi$ restricts to an isomorphism
$R\xrightarrow{\sim} C^-.$
Moreover,
\[
\mathcal O_X(-E)|_R
\simeq \mathcal I_{C/Y}\cdot\mathcal O_{C^-}
\simeq \mathcal O_{C^-}(-C\cap C^-),
\]
and therefore $E.R=\operatorname{length}(C\cap C^-)=6. $ Since $H.R=3$, we obtain $(-K_X).R=(2H-E).R=6-6=0.$ Thus $\varphi_{|-K_X|}$ contracts $R$. Let
\[
X\xrightarrow{\psi}X'\xrightarrow{f}\varphi_{|-K_X|}(X)
\]
be its Stein factorization. Since $-K_X$ is big,
$\varphi_{|-K_X|}$ is generically finite onto its image, and hence
$\dim X'=3$. Since $\psi$ has connected fibers and is generically
finite, it is birational. Thus $\psi\colon X\longrightarrow X'$ is a nontrivial birational anticanonical contraction.

\subsection{The K3 lattice and smallness}\label{sec:smallness}

It remains to show that the anticanonical morphism does not contract a divisor. We use the K3 surface constructed above.

On $Y$, consider the divisor $\overline S :=\{t=w-Q=0\}$. By \eqref{eq:keyrelation}, the equation $t=0$ on $Y$ forces $G=0$. Conversely, if $G=0$ and $w=Q$, then the equation of $Y$ is satisfied. Thus projection $q$ identifies $\overline S$ with $S_0$:
$$
 \overline S\xrightarrow{\sim}S_0.
$$
In particular $\overline S$ is smooth. Since $t$ has weight $2$, $\overline S\sim2H=-K_Y$. It contains $C$ with multiplicity one.

Let $S\subset X$ be the strict transform of $\overline S$. Since $C$
is a Cartier divisor on the smooth surface $\overline S$, the blow-up
of $\overline S$ along $C$ is isomorphic to $\overline S$. Hence $S\simeq \overline S\simeq S_0,$ and $S\sim 2H-E=-K_X.$ Under this identification, $H|_S=h$ and $E|_S=\Gamma.$
Indeed,
\[
\mathcal O_X(-E)|_S\simeq \mathcal I_{C/\overline S}
\simeq \mathcal O_{\overline S}(-C).
\]
Consequently, $(-K_X)|_S=2h-\Gamma.$ Using
$h^2=4, h\cdot\Gamma=3, \Gamma^2=-2,$ we obtain
\[
((-K_X)|_S)^2=(2h-\Gamma)^2=16-12-2=2.
\]
Since $-K_X$ is nef, $(-K_X)|_S$ is nef.

\begin{lemma}\label{lem:Lampleness}
The divisor $(-K_X)|_S=2h-\Gamma$ is ample on $S$.
\end{lemma}

\begin{proof}
Let $D_S=xh+y\Gamma\in\Pic(S)$. Using \eqref{eq:lattice},
\[
(-K_X)|_S\cdot D_S=(2h-\Gamma)\cdot(xh+y\Gamma)=5x+8y,
\]
so $(-K_X)|_S^\perp\cap\Pic(S)$ is generated by $8h-5\Gamma$, with
\[
 (8h-5\Gamma)^2=64\cdot4-80\cdot3+25(-2)=-34.
\]
Every nonzero class orthogonal to $(-K_X)|_S$ therefore has square $-34k^2$, so $(-K_X)|_S^\perp$
contains no class of square $-2$. Since $((-K_X)|_S)^2=2>0$, the standard ampleness
criterion for K3 surfaces (a nef class of positive square is ample unless
it is orthogonal to an effective class of square $-2$; see e.g.\
\cite[Ch.~8]{Huy16}) shows that $(-K_X)|_S$ is ample.
\end{proof}

\begin{proposition}\label{prop:small}
The anticanonical contraction $\psi\colon X\to X'$ is small and nontrivial.
\end{proposition}

\begin{proof}
Nontriviality follows from $(-K_X).R=0$. We show that no divisor is contracted.

Suppose $D\subset X$ is an irreducible divisor contracted by $\psi$. Since
$|-K_X|$ is base point free, $\varphi_{|-K_X|}$ factors as $f\circ\psi$ with
$f\colon X'\to\mathbb{P}^N$ finite, so $-K_X=\psi^*A'$ with
$A'=f^*\mathcal{O}_{\mathbb{P}^N}(1)$ ample on $X'$. As $\dim\psi(D)\le1$,
the divisor $(-K_X)|_D$ is pulled back from a variety of dimension at most
one, so $(-K_X)^2\cdot D=\bigl((-K_X)|_D\bigr)^2=0$.

Write $D\sim aH+bE$ for $a,b\in\mathbb Z$. Since $D$ is contracted by $\psi$,
\[
0=(-K_X)^2\cdot D=5a+8b.
\]
Hence $D\sim m(8H-5E)$ for some nonzero integer $m$.

Since $S\sim -K_X$, we have $(-K_X)^2\cdot S=(-K_X)^3=2,$ so $S$ is not contracted by $\psi$. In particular, $D\neq S$, and
therefore $S$ is not a component of $D$. Restricting to $S$ and using $H|_S=h$ and $E|_S=\Gamma,$ gives $D|_S\sim m(8h-5\Gamma).$
This class is nonzero. Since $D$ is effective and does not contain $S$, its restriction to $S$ defines an effective divisor. Moreover,
this divisor is nonzero, since otherwise
$\mathcal O_S(D|_S)\simeq\mathcal O_S$, contrary to the fact that
$m(8h-5\Gamma)\neq0$ in $\operatorname{Pic}(S)$. Thus
\[
(-K_X)|_S\cdot D|_S=(-K_X)^2\cdot D=0,
\]
contradicting the ampleness of $(-K_X)|_S$. Hence $\psi$ contracts no divisor.
Since it contracts $R$, the contraction $\psi$ is small and
nontrivial.
\end{proof}

\section{Identification with E1-E1 Case No. 28}\label{sec:identification}

Since $X$ is smooth and $\psi$ is a small $K$-trivial contraction, by
\cite{KM98} there is a flop
\[
\chi\colon X \dashrightarrow X^+ ,
\]
with $\psi^+\colon X^+ \to X'$ the corresponding small contraction. The
threefold $X^+$ is again smooth, has Picard rank two, and $-K_{X^+}$ is nef
and big. Its second extremal ray gives a $K$-negative contraction
\[
\pi^+\colon X^+ \longrightarrow Y^+ ,
\]
which we now identify. The key point is that the flop is realized by a
birational involution of $X$.

\begin{theorem}\label{thm:involution}
The Sarkisov link of Case~28 is symmetric. More precisely, the anticanonical
morphism of $X$ factors through a finite double cover $f\colon X' \to \PP^3$
with covering involution $\tau$, and the birational involution
\[
\sigma := \psi^{-1}\circ\tau\circ\psi\colon X \dashrightarrow X
\]
realizes the flop of $\psi$: there is an isomorphism $g\colon X^+ \to X$ with $\tau\circ\psi\circ g = \psi^+$ and $g\circ\chi = \sigma$. In particular $\pi^+ = \pi\circ g$ up to an isomorphism $Y^+ \simeq Y$, and under this
identification the Sarkisov link $\Phi = \pi^+\circ\chi\circ\pi^{-1}\colon Y \dashrightarrow Y$ is the birational involution $\pi\circ\sigma\circ\pi^{-1}$.
\end{theorem}

\begin{proof}
By Proposition~\ref{prop:weakfano}, $-K_X$ is nef and big with
$(-K_X)^3 = 2$, so Riemann-Roch and Kawamata-Viehweg vanishing give
$h^0(X,-K_X) = \tfrac12(-K_X)^3 + 3 = 4$. The sections $t, Q_1, Q_2, Q_3$
of \eqref{eq:idealC} are linearly independent sections of $-K_X = 2H - E$,
hence form a basis of $H^0(X,-K_X)$, and the anticanonical morphism is
\[
\varphi = (t : Q_1 : Q_2 : Q_3)\colon X \longrightarrow \PP^3 .
\]
Since $-K_X$ is big, $\varphi$ is generically finite onto $\PP^3$, of degree
$(-K_X)^3 = 2$. Let $X \xrightarrow{\ \psi\ } X' \xrightarrow{\ f\ } \PP^3$
be the Stein factorization. By Proposition~\ref{prop:small}, $\psi$ is small
and birational, while $f$ is finite of degree two. As $X'$ is normal, $f$
has a covering involution $\tau\colon X' \to X'$ with $f\circ\tau = f$, and
\[
\sigma = \psi^{-1}\circ\tau\circ\psi\colon X \dashrightarrow X
\]
is a birational involution with $\psi\circ\sigma = \tau\circ\psi$. Since
$\psi$ is small, $\sigma$ is an isomorphism in codimension one, so
$\sigma^*$ is an involution of $\Pic(X)$. Since $-K_X = \psi^*M$ with
$M = f^*\OO_{\PP^3}(1)$ and $\tau^*M = M$, we have
$\sigma^*(-K_X) = -K_X$.

We compute $\sigma^*E$ by restricting to the K3 surface $\Sbar$ of
Section~\ref{sec:smallness}. Since $t$ is the section of $-K_X$ defining
$\Sbar$, we have $\Sbar = \varphi^*P$ for the plane
$P = \{x_0 = 0\} \subset \PP^3$. The surface $\Sbar$ is disjoint from the
exceptional locus of $\psi$: if $R'$ is a curve contracted by $\psi$, then
$(-K_X)\cdot R' = 0$, so $R' \not\subset \Sbar$ because
$L := (-K_X)|_{\Sbar}$ is ample by Lemma~\ref{lem:Lampleness}, and then
$\Sbar\cdot R' = 0$ forces $R' \cap \Sbar = \emptyset$. Hence $\sigma$ is an
isomorphism near $\Sbar$, and $\sigma(\Sbar) = \Sbar$ since
$\psi(\Sbar) = f^{-1}(P)$ is $\tau$-invariant. The restriction
$\sigma|_{\Sbar}$ is therefore the covering involution $\iota$ of the
degree-two morphism $p = \varphi|_{\Sbar}\colon \Sbar \to P \simeq \PP^2$
defined by $|L|$.

The transfer argument for the finite quotient
$p\colon \Sbar \to \Sbar/\langle\iota\rangle = \PP^2$ gives
$H^2(\Sbar,\mathbb{Q})^{\iota} \simeq H^2(\PP^2,\mathbb{Q})$, generated by the pullback
$L$ of the hyperplane class. As $\iota^*$ is an involutive isometry of
$\Pic(\Sbar)$, its $(+1)$-eigenspace is $\mathbb{Q} L$ and its $(-1)$-eigenspace
is $L^\perp = \mathbb{Q} v$, where $v = 8h - 5\Gamma$ by Lemma~\ref{lem:Lampleness}.
Thus $\iota^*L = L$ and $\iota^*v = -v$. Since $\Gamma = 4L - v$,
$\iota^*\Gamma = 4L + v = 16h - 9\Gamma$. The restriction map $\Pic(X) \to \Pic(\Sbar)$ is an isomorphism sending $H$ to $h$ and $E$ to $\Gamma$, and $(\sigma^*E)|_{\Sbar} = \iota^*(E|_{\Sbar}) = 16h - 9\Gamma
= \bigl(8(-K_X) - E\bigr)|_{\Sbar}$. Therefore
\begin{equation}\label{eq:sigmaE}
\sigma^*E = 8(-K_X) - E,
\qquad
\sigma^*H = \tfrac12\bigl(-K_X + \sigma^*E\bigr) = 5(-K_X) - H .
\end{equation}

We now show that $\sigma$ is the flop of $\psi$. The divisor $H$ is nef and
$\psi$-ample: a curve contracted by $\psi$ and orthogonal to $H$ would be
orthogonal to both $-K_X$ and $H$, which span $\Pic(X)\otimes\mathbb{Q}$, and
hence numerically trivial. Since $-K_X$ is $\psi$-trivial,
\eqref{eq:sigmaE} gives $(\sigma^*H).R' = -H.R' < 0$ for every
curve $R'$ contracted by $\psi$, equivalently by $\tau\circ\psi$. Thus $X$,
regarded as an $X'$-scheme via $\tau\circ\psi$, is a small projective
modification of $X$ over $X'$ on which the strict transform
$\sigma_*H = \sigma^*H$ of the $\psi$-ample divisor $H$ is relatively
anti-ample. This is the defining property of the flop of $\psi$, which is
unique, namely $X^+ = \Proj_{X'}\bigoplus_{m\ge 0}\psi_*\OO_X(-mH)$; see
\cite[\S6.1]{KM98} or \cite{Kollar89}. Hence there is an isomorphism
$g\colon X^+ \to X$ with $\tau\circ\psi\circ g = \psi^+$ and
$g\circ\chi = \sigma$.

Under $g$ the nef cone of $X^+$ is identified with
$\Nef(X) = \langle -K_X, H\rangle$, so the second extremal contraction
$\pi^+\colon X^+ \to Y^+$ is $\pi\circ g$ up to an isomorphism
$Y^+ \simeq Y$. Identifying $Y^+$ with $Y$ in this way,
\[
\Phi = \pi^+\circ\chi\circ\pi^{-1} = \pi\circ\sigma\circ\pi^{-1},
\]
and $\sigma^2 = \mathrm{id}$ gives $\Phi^2 = \mathrm{id}_Y$.
\end{proof}

%%%%%%%%%%%%%%%%%%%%%%%%%%%%%%%%%%%%%%%%%%%%%%%%%%%%%%%%%%%%%%%%

\begin{corollary}\label{cor:invariants}
The link $\Phi$ has the numerical invariants of E1-E1 Case~28: the
contraction $\pi^+$ is of type E1 with
\[
(-K_{Y^+})^3 = 16, \qquad (r^+,d^+,g^+) = (2,3,0),
\]
and
\[
\alpha = 8, \qquad \beta = -1, \qquad e = 552, \qquad \frac{e}{r^3} = 69 .
\]
\end{corollary}

\begin{proof}
By Theorem~\ref{thm:involution}, $\pi^+ = \pi\circ g$ under the
identification $Y^+ \simeq Y$, so $\pi^+$ is the blow-up of the smooth
rational cubic curve $C \subset Y = V_2$. Hence $\pi^+$ is of type E1,
$(-K_{Y^+})^3 = (-K_Y)^3 = 16$ by \eqref{PicY}, and
$(r^+,d^+,g^+) = (r,d,g) = (2,3,0)$ by \eqref{PicY} and \eqref{eq:Cdata}.

The exceptional divisor of $\pi^+ = \pi\circ g$ is $E^+ = g^*E$. Since
$\chi$ and $g$ are isomorphisms in codimension one, the strict transform of
$E^+$ on $X$ under $\chi^{-1}$ is
\[
D = \chi^*E^+ = \chi^*g^*E = (g\circ\chi)^*E = \sigma^*E = 8(-K_X) - E
\]
by \eqref{eq:sigmaE}. Thus $\alpha = 8$ and $\beta = -1$.

Finally, the strict transform $\Et = \chi_*E$ of $E$ on $X^+$ corresponds
under the isomorphism $g$ to $g_*\Et = (g\circ\chi)_*E = \sigma_*E =
\sigma^*E = 8(-K_X) - E$, and triple intersection numbers are preserved by
the isomorphism $g$. Using \eqref{eq:-K_X3}-\eqref{eq:-K_XE2} and $E^3 = -4$ from \eqref{eq:HEtable},
\begin{align*}
\Et^3 &= \bigl(8(-K_X) - E\bigr)^3 \\
&= 512(-K_X)^3 - 192(-K_X)^2. E + 24(-K_X).E^2 - E^3 \\
&= 512\cdot 2 - 192\cdot 8 + 24\cdot(-2) - (-4) = -556 .
\end{align*}
Hence $e = E^3 - \Et^3 = -4 - (-556) = 552$, and since $r = 2$,
$e/r^3 = 552/8 = 69$.
\end{proof}

\begin{proof}[Proof of Theorem~\ref{thm:main}]
By Propositions~\ref{prop:weakfano} and \ref{prop:small}, $X = \Bl_C(Y)$ is a smooth weak Fano threefold of Picard number two with small anticanonical morphism, and $\pi\colon X \to Y = V_2$ is the blow-up of a smooth rational cubic, so $(-K_Y)^3 = 16$ and $(r,d,g) = (2,3,0)$. By
Theorem~\ref{thm:involution} and Corollary~\ref{cor:invariants}, the flop
$\chi$ exists, the contraction on the other side is again the blow-up of a
smooth rational cubic curve on $V_2$, and the invariants of the link are those in \eqref{eq:case28data}-\eqref{eq:table-defect}. These are precisely the invariants of Case~28.
\end{proof}
 
\begin{remark}\label{rem:tables}
The invariants in Corollary~\ref{cor:invariants} are obtained directly from the geometry of the involution $\sigma$, without consulting the
classification tables. They are consistent with those tables: none of the
cases classified in \cite{JPR07}, nor the non-E1-E1 cases of \cite{CM13},
has $(-K_X)^3 = 2$, $(-K_Y)^3 = 16$, $(r,d,g) = (2,3,0)$, and Case~28 is
the unique row of the E1-E1 table of \cite{CM25} with these left-hand
invariants. The value $\beta = -1$ also agrees with the general relation
$\beta = -r^+/r$ for E1-E1 links \cite[Lemma~2.2]{CM13}, and the value
$\alpha = 8$ with the relation $(-K_X)^2.D = (-K_{X^+})^2.E^+ = 8$, which holds because intersection numbers involving at least one factor of the anticanonical divisor are preserved across the $K$-trivial flop.
\end{remark}

\section{Concluding Remarks}

The methods of \cite{ACM17} use projective models of rank-one Fano
threefolds as complete intersections or linear sections in homogeneous
varieties. As noted in \cite{CM25}, these methods do not apply directly
to Case 28 because the degree-two del Pezzo threefold $V_2$ is naturally
realized as a quartic double solid. The construction above uses this
double-cover model directly. In particular, the splitting $q^{-1}(\Gamma)=C\cup C^- $ provides both the center of the blow-up and a $K_X$-trivial curve on the resulting threefold, and the covering involution of the anticanonical double cover $X' \to \PP^3$ provides the flop.\\

The E1-E1 table in \cite{CM25} lists Case 28 as open. Theorem~\ref{thm:main} proves that this case is geometrically realizable. Thus, among the four E1-E1 cases left open in \cite{CM25}, the remaining cases are 59, 61, and 80.
\newpage

\section*{Acknowledgments}
The authors express their gratitude to V.V. Shokurov for his lifelong mentorship and guidance, and to V. Cheltsov for proposing this alternative approach to the construction. Special thanks are also due to the hosts of the SAGA$^2$ Conference at the University of French Polynesia for organizing a highly successful and welcoming event.

\end{document}